\documentclass{article}

\usepackage{arxiv}

\usepackage[utf8]{inputenc} 
\usepackage[T1]{fontenc}    
\usepackage{hyperref}       
\usepackage{url}            
\usepackage{booktabs}       
\usepackage{amsfonts}       
\usepackage{nicefrac}       
\usepackage{microtype}      
\usepackage{lipsum}
\usepackage{graphicx}
\usepackage{epstopdf}

\usepackage{amsmath}

\usepackage{color}
\usepackage{amsthm}
\theoremstyle{definition}

\newtheorem{prop}{Proposition}[section]
\newtheorem{lemma}{Lemma}[section]
\newtheorem{thm}{Theorem}[section]

\newtheorem{assumption}{Assumption}[section]

\newtheorem{example}{Example}[section]
\theoremstyle{remark}

\title{The backward Euler-Maruyama method for invariant measures of stochastic differential equations with super-linear coefficients}

\author{
 Wei Liu \\
 Department of Mathematics, \\
 Shanghai Normal University \\
 Shanghai, 200234, China\\
 \texttt{weiliu@shnu.edu.cn; lwbvb@hotmail.com} \\
 \AND
  Xuerong Mao \\
  Department of Mathematics and Statistics,\\
  University of Strathclyde\\
  Glasgow,  G1 1XH, UK\\ 
  \texttt{x.mao@strath.ac.uk}  \\
\AND
  Yue Wu \\
  Department of Mathematics and Statistics,\\
  University of Strathclyde\\
  Glasgow,  G1 1XH, UK\\ 
  \texttt{yue.wu@strath.ac.uk}  \\
}

\begin{document}
\maketitle

\begin{abstract}
The backward Euler-Maruyama (BEM) method is employed to approximate the invariant measure of stochastic differential equations, where both the drift and the diffusion coefficient are allowed to grow super-linearly. The existence and uniqueness of the invariant measure of the numerical solution generated by the BEM method are proved and the convergence of the numerical invariant measure to the underlying one is shown. Simulations are provided to illustrate the theoretical results and demonstrate the application of our results in the area of system control.
\end{abstract}


\section{Introduction}

Invariant measure is one of essential properties of stochastic differential equations (SDEs), when long time behaviours of SDEs are investigated, such as the persistence for biology and epidemic SDE models \cite{Allen2007,Mao2008}. However, the explicit forms of neither the true solutions nor the invariant measures to SDEs are easily found. Therefore, numerical methods become extremely important when SDE models are applied in practice.   
\par
For SDEs of the It\^o form 
\begin{align}
  \label{eq:introSDE}
  \begin{cases}
    \mathrm{d}{X_t} = 
    \mu(X_t)  \mathrm{d}{t}+\sigma(X_t)\mathrm{d}{W_t},& \quad \text{for } t >0,\\
    X_{0} = x\in \mathbb{R}^d,& 
  \end{cases}
\end{align}
Yuan and Mao in \cite{YM2004} studied the numerical invariant measure generated by the Euler-Maruyama (EM) method when both the coefficients $\mu (\cdot)$ and $\sigma (\cdot)$ obey the global Lipschitz condition. Under the same condition on the coefficients, Weng and Liu investigated the numerical approximation to invariant measures of SDEs by the Milstein method \cite{WL2019}. When some non-global Lipschitz terms appear in the coefficients, the backward Euler-Maruyama (BEM) method (also called the semi-implicit Euler method) and the truncated EM method were employed to handle the super-linearity. Liu and Mao in \cite{LM2015} discussed the BEM method for numerically approximating the invariant measure when the one-sided Lipschitz condition was imposed on the draft coefficient $\mu (\cdot)$ but the global Lipschitz condition was still required for the diffusion coefficient $\sigma (\cdot)$. Jiang, Weng and Liu further studied the stochastic $\theta$ method for this problem and discussed the effects of the different choices of $\theta$ on the requirements on the coefficients \cite{JWL2020}. When the constraints on the coefficients were further released, Li, Mao and Yin proposed the truncated EM method to approximate the invariant measure of the underlying SDEs \cite{LMY2019}.        
\par
In this paper, we revisit the BEM method and study the numerical approximation to invariant measures of SDEs with both the drift and diffusion coefficients containing super-linear terms. Compared with the existing work \cite{LM2015}, where only the drift coefficient was allowed to grow super-linearly, our work releases the condition on the diffusion coefficient such that the super-linear terms are also allowed. To achieve such a better result, a different technique is employed in this paper. Briefly speaking, instead of directly forming an iteration for the numerical solution of $X_t$, we construct an iteration for the numerical approximation of some linear combination of $X_t$ and $\sigma(X_t)$. It should be mentioned that this technique is inspired by \cite{AK2017}. Similar techniques were employed for the studies on the finite time convergence and the stability of the trivial solution of the BEM method \cite{CG2020} and the stochastic $\theta$ method \cite{WWD2020}, and for the study on the infinite time convergence of BEM method to the random periodic solution of the SDEs with additive noise \cite{W2022}. But, to our best knowledge, there is no exiting work on the numerical approximation to invariant measures of SDEs with the super-linear drift and diffusion coefficients by using the BEM method. 
\par
Therefore, the result obtained in this paper can be regarded as an extension to \cite{LM2015} and a complement to the study on BEM method in the aspect of numerical invariant measures. In addition, our results can support the application of theorems on stabilisation of SDEs in the distribution sense that were recently developed in \cite{LLLM2022,YHLM2022} (see Example \ref{example:2d} for the illustration).  
It should be mentioned that the numerical invariant measure of SDEs obtained in this paper could also assist in approximating the corresponding high-dimensional partial differential equation such as the high-dimensional Fokker–Planck equation. With the help of the neural network architecture, such an approach through a probabilistic representation to learn the solution of some partial differential equation could be quite efficient \cite{HJE2018}.
\par
Other approaches were also proposed and investigated for approximating invariant measures of SDEs. An incomplete list includes \cite{BSY2016,FG2018,MSH2002,Talay1990}, among many others. The BEM method, as the simplest version of implicit methods, was widely studied for many different types of stochastic equations \cite{GLM2019,MS2013,XSL2011,ZX2019,Zhou2015}. We just mention some of them here and refer the readers to the reference therein for more works.
\par
We end this introduction with some discussions on the competition between explicit and implicit methods. For stiff ordinary differential equations, implicit methods are preferred due to its good performance even with on a time grid with a large step size \cite{wanner1996}. But for its stochastic counterpart, explicit methods are also popular \cite{HJK2012,Mao2015}. Since many sample paths are usually needed to be simulated in practice, explicit methods have their advantages like simple algorithm structure, easy to implement and no need to solve nonlinear equation systems in each iterations, if simulations are conducted in some finite short intervals. For simulations of long time behaviours of SDEs, implicit methods that pose better stability properties allow large step-sizes and have low total computational costs. More interesting and detailed discussions on this topic can be found in, for example \cite{Higham2011,MT2010}.

\section{Mathematical Preliminaries}

Let $W \colon [0,\infty) \times \Omega \to \mathbb{R}^n$ be a standard Wiener process on the probability space $(\Omega, \mathcal{F}, \mathbb{P})$, with the filtration defined by $\mathcal{F}_s^{t}:=\sigma\{W_u-W_v:s<v\le u<t\}$ and $\mathcal{F}^{t}=\mathcal{F}^{t}_0=\vee_{0\le s\le t}\mathcal{F}^{t}_s$. Throughout this paper, we shall use $|\cdot|$ for the Euclidean norm and $\langle\cdot,\cdot\rangle$ for the inner product in the Euclidean space. For a vector $u$, we define $\|u\|:=\sqrt{\mathbb{E}[|u|^2]}$ and $\|u\|_p:=\sqrt[p]{\mathbb{E}[|u|^p]}$. For a matrix $B$, $\|B\|_{\rm HS}$ means its Hilbert-schmidt norm. In addition, we define $\|B\|:=\sqrt{\mathbb{E}[\|B\|^2_{\rm HS}]}$ and $\|B\|_{p}:=\sqrt[p]{\mathbb{E}[\|B\|^p_{\rm HS}]}$. Denote $a\vee b$ the larger one between scalars $a$ and $b$, and $a\wedge b$ the smaller one. The family of all probability measures on $\mathbb{R}^d$ is denoted by $\mathcal{P} (\mathbb{R}^d)$. Let $\mathcal{B}(\mathbb{R}^d)$ denote the family of all Borel sets in $\mathbb{R}^d$.
\par
The $r$-Wasserstein distance between $\mu_1, \mu_2 \in \mathcal{P} (\mathbb{R}^d)$  for any $r \in (0,1]$ is defined by 
\begin{equation*}
    \mathbb{W}_r(\mu_1,\mu_2) = \inf_{\nu \in \mathcal{C} (\mu_1, \mu_2)} \int_{\mathbb{R}^d \times \mathbb{R}^d}|Y_1-Y_2|^r \nu(\mathrm{d} Y_1, \mathrm{d} Y_2),
\end{equation*}
where $\mathcal{C} (\mu_1, \mu_2)$ denotes the set of all couplings of $\mu_1$ and $\mu_2$.
\par
Given a stochastic process $X_t$ on $(\Omega, \mathcal{F}, \mathbb{P})$, for any $t>0$ and any $B \in \mathcal{B} (\mathbb{R}^d)$ let $\mathbb{P}_t (x,B)$ be the transition probability kernel of $X_t$. A probability measure $\pi(\cdot) \in \mathcal{P} (\mathbb{R}^d)$ is called an \emph{invariant measure} of $X_t$, if
\begin{equation*}
    \pi(B) = \int_{\mathbb{R}^d} \mathbb{P}_t (x,B) \pi(\mathrm{d} x)
\end{equation*}
holds for any $t > 0$ and any $B \in \mathcal{B} (\mathbb{R}^d)$.

In this paper, we are interested in the stationary measure of the solution to the $\mathbb{R}^d$-valued SDE of the form
\begin{align}
  \label{eq:SPDE}
  \begin{cases}
    \mathrm{d}{X_t} = 
    \big[-A X_t + f(X_t) \big] \mathrm{d}{t}+g(X_t)\mathrm{d}{W_t},& \quad \text{for } t >0,\\
    X_{0} = x\in \mathbb{R}^d.& 
  \end{cases}
\end{align}
We separate the drift coefficient into two parts with the emphasis on the negative linear term $-AX_t$, as it could be regarded as stabiliser term \cite{yevik2011}.
We impose several assumptions on $A$, $f$, and $g$ as follows.
\begin{assumption}
  \label{as:A}
  The linear operator $A \colon \mathbb{R}^d \to \mathbb{R}^d$ is
  self-adjoint and positive definite.
\end{assumption}

Assumption \ref{as:A} implies the existence of a positive, increasing
sequence $(\lambda_i)_{i\in \mathbb{N}} \subset \mathbb{R}$ such that 
$0<\lambda_1 \le \lambda_2 \le \ldots \lambda_d$, and of an orthonormal basis $(e_i)_{i \in [d]}$ of $ \mathbb{R}^d$ such that $A e_i = \lambda_i  e_i$ for every $i \in
[d]$, where $[d]:=\{1,\ldots,d\}$.

\begin{assumption}
  \label{as:f_Khasminskii-type}
  There exists a constant $q\in [1,\infty)$ and a positive $L$ such that 
  $$|f(u_1)-f(u_2)|\vee \|g(u_1)-g(u_2)\|_{\rm HS}\leq L(1+|u_1|^{q-1}+|u_2|^{q-1})|u_1-u_2|,$$
  for $u_1,u_2\in \mathbb{R}^d$.
\end{assumption}
It is straightforward to derive from Assumption \ref{as:f_Khasminskii-type} that
$$|f(u)|\vee \|g(u)\|_{\rm HS}\leq (2L \vee |f(0)| \vee \|g(u)\|_{\rm HS} )(1+|u|^{q})$$
for $u \in \mathbb{R}^d$.
\par

\begin{assumption}
  \label{as:fg}
  The mappings $f \colon \mathbb{R}^d \to \mathbb{R}^d$  and $g \colon \mathbb{R}^d \to \mathbb{R}^{d\times n}$ are continuous.
  Moreover, there exist $c, c_1,c_2 \in
  (0,\infty)$ and $l_1\geq 2,l_2 \geq 4q-3$ such that 
  \begin{align*}
   & 2\langle u_1-u_2, f(u_1) - f(u_2) \rangle+l_1\|g(u_1)-g(u_2)\|^2_{\rm HS} \leq c |u_1 - u_2|^2\\
    &  2\langle u, f(u) \rangle +l_2\|g(u)\|^2_{\rm HS}\le c_1 +c_2 |u|^2
  \end{align*}
  for all $u,u_1, u_2 \in \mathbb{R}^d$.
\end{assumption}
It is well known that under these assumptions the \emph{ solution} $X_{t} \colon [0,\infty) \times
\Omega \to \mathbb{R}^d$ to \eqref{eq:SPDE} is uniquely determined \cite{Mao2008}. 
With an additional assumption imposed as below,
\par
\begin{assumption}
  \label{as:f_constant}
  $c\vee c_2 <\lambda_1,$
\end{assumption}
we can show in Section \ref{sec:sol} (as in Proposition \ref{prop:use_bound2} and Proposition \ref{prop:use_bound}) the solution to SDE \eqref{eq:SPDE} is uniformly bounded in $L^p$ sense, i.e.,
\begin{align*}
    \sup_{t\geq 0}\|X_t \|_p^p<\infty,
\end{align*}
and is H\"older-continuous in the temporal variable.

Now, we give a brief revisit to the well-known BEM method.

Let us fix an equidistant partition $\mathcal{T}^h:=\{jh,\ j\in \mathbb{N} \}$ with stepsize $h\in (0,1)$. Note that $\mathcal{T}^h$ stretch along the positive real line because we are dealing with an infinite time horizon problem. Then to simulate the solution to \eqref{eq:SPDE} starting at $0$, the \emph{backward Euler-Maruyama method} on $\mathcal{T}^h$ is given by the recursion 
\begin{align}
  \label{eq:RandM}
  \begin{split}
    \hat{X}_{(j+1)h}=& \hat{X}_{jh} - Ah\hat{X}_{(j+1)h}+h f\big( \hat{X}_{(j+1)h} \big)+ 
    g(\hat{X}_{jh})\Delta W_{jh}
  \end{split}
\end{align}
for all $j \in \mathbb{N}$, where the initial value $\hat{X}_{0}  =x$, and $\Delta W_{jh}:=W_{(j+1)h}-W_{jh}$.

The implementation of \eqref{eq:RandM} requires solving a nonlinear equation at each iteration. The well-posedness of the difference equation \eqref{eq:RandM} is proved in the next lemma.

\begin{lemma}
Let Assumptions \ref{as:fg} and \ref{as:f_constant} hold. Then the BEM method is well defined.
\end{lemma}
\begin{proof}
For any $N \in \mathbb{N}$, rewrite the BEM method \eqref{eq:RandM} into
\begin{equation*}
    \hat{X}_{(N+1)h} + Ah\hat{X}_{(N+1)h} - h f\big( \hat{X}_{(N+1)h} \big) = \hat{X}_{Nh} + 
    g(\hat{X}_{Nh})\Delta W_{Nh}.
\end{equation*}
Define $G(u) = u + Ahu -hf(u) $ for $u \in \mathbb{R}^d$. By Assumption \ref{as:fg}, we have 
\begin{equation*}
    2\langle u_1-u_2, f(u_1) - f(u_2) \rangle \leq c |u_1 - u_2|^2
\end{equation*}
for all $u_1, u_2 \in \mathbb{R}^d$. Then, it is straightforward to see
\begin{equation*}
    \langle u_1-u_2, G(u_1) - G(u_2) \rangle \geq \big(1+\lambda_1h - \frac{ch}{2}\big) |u_1 - u_2|^2.
\end{equation*}
Due to Assumption \ref{as:f_constant}, $1+\lambda_1h - ch/2 > 0$ holds for all $h>0$, which means that $G(\cdot)$ is monotonic. So $G(\cdot)$ has its inverse function $G^{-1} (\cdot): \mathbb{R}^d \rightarrow \mathbb{R}^d$ such that for any $N \in \mathbb{N}$
\begin{equation*}
    \hat{X}_{(N+1)h} = G^{-1} \big( \hat{X}_{Nh} + 
    g(\hat{X}_{Nh})\Delta W_{Nh} \big).
\end{equation*}
That is to say, for any $N \in \mathbb{N}$ the unique $\hat{X}_{(N+1)h}$ can always be found for the given $\hat{X}_{Nh} + g(\hat{X}_{Nh})\Delta W_{Nh}$, which completes the proof.
\end{proof}

To explore the invariant measure of the numerical solution, we introduce some more notations. For any $j \in \mathbb{N}$ and any $B \in \mathcal{B} (\mathbb{R}^d)$, let $\hat{\mathbb{P}}_{jh} (x,B)$ be the transition probability kernel of $\hat{X}_{jh}$. A probability measure $\hat{\pi}(\cdot) \in \mathcal{P} (\mathbb{R}^d)$ is called an \emph{invariant measure} of $\hat{X}_{jh}$, if
\begin{equation*}
    \hat{\pi}(B) = \int_{\mathbb{R}^d} \hat{\mathbb{P}}_{jh} (x,B) \hat{\pi}(\mathrm{d} x)
\end{equation*}
holds for any integer $j \in \mathbb{N}$ and any $B \in \mathcal{B} (\mathbb{R}^d)$.

We end up this section by pointing out {\bf the crucial equality} for analysis of the backward Euler-Maruyama in our paper. For any $a,b \in \mathbb{R}^d$, the equality 
\begin{equation}\label{eqn:eqa}
    |b|^2-|a|^2+|b-a|^2=2\langle b-a, b\rangle
\end{equation}
holds.

\section{Some properties of the underlying solution}\label{sec:sol}
In this section, we mainly explore properties of the solution to \eqref{eq:SPDE} for analysis later. 
  
The first property we will show is the uniform boundedness for the $p$-th moment of the SDE solution.
\begin{prop}\label{prop:use_bound2} Suppose that Assumptions \ref{as:A} to \ref{as:f_constant} hold. Then, for any $p \in [2, l_2+1]$ the solution to \eqref{eq:SPDE} satisfies
\begin{equation} 
   \sup_{t\geq 0}\|X_{t}\|^p_p<\infty.
\end{equation}
\end{prop}
\begin{proof}
Due to Assumption \ref{as:f_constant}, we have $c_2 < 2 \lambda_1$. Then, let $\epsilon$ be a sufficiently small positive number such that $p\lambda_1 - 0.5p c_2 -\epsilon >0$.
By the It\^o formula,
\begin{equation*}
\mathbb{E}[ e^{\epsilon t } |X_t|^p ]  \leq |x|^p
+ \mathbb{E} \int_0^t e^{\epsilon s} \left[ -(p\lambda_1-\epsilon) |X_s|^p + 0.5p |X_s|^{p-2} \left( 2 \langle X_x,f(X_s)\rangle
+(p-1)\|g(X_s)\|^2_{\rm HS} \right) \right]\mathrm{d} s.
\end{equation*}
As $p - 1 \leq l_2$,  Assumption \ref{as:fg} indicates
\begin{align*}
\mathbb{E}[ e^{\epsilon t } |X_t|^p ] & \le |x|^p
+ \mathbb{E} \int_0^t e^{\epsilon s} \Big[ -(p\lambda_1-\epsilon) |X_s|^p
+ 0.5p |X_s|^{p-2} (c_1+c_2 |X_s|^2 ) \Big]\mathrm{d}s \\
& = |x|^p
+ \mathbb{E} \int_0^t e^{\epsilon s} \Big[ -(p\lambda_1 - 0.5p c_2 -\epsilon) |X_s|^p
+ 0.5pc_1 |X_s|^{p-2} \Big]\mathrm{d} s
\end{align*}
Since $ p\lambda_1 - 0.5p c_2 -\epsilon >0$, we know that the polynomial $ -(p\lambda_1 - 0.5p c_2 -\epsilon) |X_s|^p + 0.5pc_1 |X_s|^{p-2} $ is always bounded by a positive number almost surely for any $|X_s| \in \mathbb{R}$. Denote the upper bound by $K$. Hence
\begin{equation*}
\mathbb{E}[ e^{\epsilon t } |X_t|^p ] \le |x|^p + \int_0^t e^{\epsilon s} K \mathrm{d}s
\le |x|^p + (K/\epsilon)e^{\epsilon t},
\end{equation*}
which implies
\begin{equation*}
\mathbb{E}[ |X_t|^p ] \le |x|^p + (K/\epsilon), \quad \forall t\ge 0.
\end{equation*}
Therefore, the proof is completed.
\end{proof}
Following a similar argument as in Proposition 5.4 and 5.5 \cite{kruse2016}, we can easily get the following bounds.
\begin{prop}\label{prop:use_bound} Suppose that Assumptions \ref{as:A} to \ref{as:f_constant} hold, then there exists a positive constant $C_{q,A,f}$ which depends on $q$, $d$, $A$,$C_f$ only, such that
\begin{align}
    \|X_{t_1}-X_{t_2}\| \leq C_{q,A,f}\big(1+\sup_{t\geq 0}\|X_{t}\|^q_{2q}\big)|t_2-t_1|^{\frac{1}{2}},
\end{align}
for all $t_1,t_2\geq 0$. Moreover, 
\begin{align}
    \begin{split}
        &\int_{t_1}^{t_2}\big\|A\big(X_{s}-X_{t_4}\big)+f(X_{s})-f(X_{t_4})\big)\big\|\mathrm{d}s\\
        &\leq C_{q,A,f}\big(1+\sup_{t\geq 0}\|X_{t}\|^{2q-1}_{4q-2}\big)|t_2-t_1|^{\frac{3}{2}},
    \end{split}
\end{align}
for all $t_3,t_4\in [t_1,t_2]$.
\end{prop}
The next theorem states that the underlying solution admits a unique invariant measure. With the help of Propositions \ref{prop:use_bound2} and \ref{prop:use_bound}, the following theorem can be proved by following the same approach as the proof of Theorem 2.3 in \cite{BSY2016} or Theorem 7.4 in \cite{LMY2019}. So we omit the proof here.
\begin{thm}\label{thm:trueinvmea}
Suppose that Assumptions \ref{as:A} to \ref{as:f_constant} hold. Then the solution to \eqref{eq:SPDE} converges in the Wasserstein distance to a unique invariant measure $\pi \in \mathcal{P} (\mathbb{R}^d)$ with some exponential rate $\xi_2 >0 $.
\end{thm}

\section{Main results}\label{sec: num1}

In this section we will prove that the BEM method \eqref{eq:RandM} uniquely admits an invariant measure with the help of two lemmas, and show the order of convergence of the invariant measure of the BEM to the invariant measure of our target SDE \eqref{eq:SPDE}. We present our three main theorems as follows. Proofs of them are postponed, after some more preparations being given.

The first main result in our paper states the existence and uniqueness of the invariant measure of the numerical solution generated by the BEM method.
\begin{thm}\label{thm:main2}
Under Assumptions \ref{as:A}, \ref{as:fg} and \ref{as:f_constant}, for any $h\in(0,1)$ satisfying 
$$h< h^{*}:=\frac{l_2-1}{2\lambda_1-2c_2}\wedge \frac{l_1-1}{2\lambda_1-2c},$$ 
the backward Euler-Maruyama method \eqref{eq:RandM} converges in the Wasserstein distance to a unique invariant measure $\hat{\pi} \in \mathcal{P} (\mathbb{R}^d)$ with some exponential rate $\xi_1 >0 $ on $\mathcal{T}^h$.
 \end{thm}

The next theorem states the strong convergence of the BEM method with the rate of $1/2$. This result looks similar to that in \cite{WWD2020} by setting $\theta = 1$ there. But, it should be noted that our assumptions are stronger than those in \cite{WWD2020}. So the strong convergence is uniform in our case, i.e. the constant $C$ in \eqref{eq:error} is independent of $t$.
\begin{thm}\label{thm:error} 
Under Assumption \ref{as:A} to Assumption \ref{as:f_constant} and for $h$ satisfying
$$h< h^{**}:= \frac{l_2-1}{2(\lambda_1-c_2)}\wedge \frac{l_1-2}{4(\lambda_1-c)},$$
there exists a constant $C$ that depends on $q,A,f,g$ and $d$ such that the backward Euler-Maruyama method \eqref{eq:RandM} approximates the true solution of \eqref{eq:RandM} on $\mathcal{T}^h$ with  
\begin{align}\label{eq:error}
 \sup_{N}\big\|X_{Nh}-\hat{X}_{Nh}\big\|\leq C h^{1/2}.
\end{align}
\end{thm}

The final main theorem states the convergence of the numerical invariant measure to the underlying one with the rate of 1/2.
\begin{thm}\label{thm:numinvconv}
Suppose that all the assumptions in Theorem \ref{thm:main2} and \ref{thm:error} hold, then the numerical invariant measure $\hat{\pi}$ converges to the underlying invariant measure $\pi$ in the Wasserstein distance, that is for any $h \in (0,h^{*}\wedge h^{**})$
\begin{equation*}
    \mathbb{W}_r\left(\hat{\pi},\pi \right) = \mathcal{O} \left( h^{r/2} \right)
\end{equation*}
holds for any $r \in (0,1]$.
\end{thm}

\subsection{Two properties of the numerical solution}
The next Lemma claims that there is a uniform bound for the second moment of the numerical solution under necessary assumptions.
\begin{lemma}\label{lem:boundedness}
Under Assumptions \ref{as:A}, \ref{as:fg} and \ref{as:f_constant}, for any $h\in (0,1)$ satisfying $$h\leq \frac{l_2-1}{2(\lambda_1-c_2)},$$
it holds for the BEM method \eqref{eq:RandM} on $\mathcal{T}^h$ that
\begin{equation} 
    \|\hat{X}_{Nh}\|^2< |x|^2 + \|g(x)\|^2_{\rm HS} +  \frac{c_1}{\lambda_1-c_2}
\end{equation}
for all $N\in\mathbb{N}$, where $x$ is the initial data.
\end{lemma}
\begin{proof}
First note that from \eqref{eqn:eqa} for any $N\in \mathbb{N}$ we have that
\begin{align}
|\hat{X}_{Nh}|^2-|\hat{X}_{(N-1)h}|^2+|\hat{X}_{Nh}-\hat{X}_{(N-1)h}|^2=2\langle \hat{X}_{Nh}-\hat{X}_{(N-1)h},\hat{X}_{Nh} \rangle.
\end{align}
From \eqref{eq:RandM} we have that
\begin{align}\label{eq:bound_eq}
       2\langle \hat{X}_{Nh}-\hat{X}_{(N-1)h},\hat{X}_{Nh} \rangle&=-2h\langle A\hat{X}_{Nh},\hat{X}_{Nh} \rangle+2h\langle f(\hat{X}_{Nh}\big),\hat{X}_{Nh}
   \rangle +2\langle g\big(\hat{X}_{(N-1)h}\big)\Delta W_{(N-1)h},\hat{X}_{Nh}\rangle. 
\end{align}
Note that $\mathbb{E}\langle g\big(\hat{X}_{(N-1)h}\big)\Delta W_{(N-1)h},\hat{X}_{(N-1)h}\rangle=0$. 

Taking the expectation of both sides of \eqref{eq:bound_eq} and making use of Assumption \ref{as:fg} give
         \begin{align*}
          &\|\hat{X}_{Nh}\|^2-\|\hat{X}_{(N-1)h}\|^2+\|\hat{X}_{Nh}-\hat{X}_{(N-1)h}\|^2=2\mathbb{E}\langle \hat{X}_{Nh}-\hat{X}_{(N-1)h},\hat{X}_{Nh} \rangle\\
  &\leq -2h\mathbb{E} \langle (A-c_2I)\hat{X}_{Nh},\hat{X}_{Nh} \rangle-l_2h\|g\big(\hat{X}_{Nh}\big)\|^2\\
  &\qquad +2hc_1+h\|g\big(\hat{X}_{(N-1)h}\big)\|^2+\|\hat{X}_{Nh}-\hat{X}_{(N-1)h}\|^2. 
\end{align*}
Then cancelling the same term on both side gives
\begin{align*}
   (1+2h (\lambda_1-c_2)) \|\hat{X}_{Nh}\|^2+l_2h\|g\big(\hat{X}_{Nh}\big)\big\|^2
    \leq 2hc_1+h\big\|g\big(\hat{X}_{(N-1)h}\big)\big\|^2+\|\hat{X}_{(N-1)h}\|^2.
\end{align*}
Choose $h$ such that $(1+2h (\lambda_1-c_2))\leq l_2$ and let $\alpha:=\frac{c_1}{\lambda_1-c_2}$. Rearranging the terms above gives
\begin{align}
\begin{split}
        &\big(1+2h(\lambda_1-c_2)\big)\big(\|\hat{X}_{Nh}\|^2+h\|g\big(\hat{X}_{Nh}\big)\big\|^2-\alpha\big)\\
    &\leq \|\hat{X}_{(N-1)h}\|^2+h\|g\big(\hat{X}_{(N-1)h}\big)\big\|^2-\alpha. 
\end{split}
\end{align}
By iteration, this leads to
\begin{align}
    \|\hat{X}_{Nh}\|^2\leq \frac{1}{\big(1+2h(\lambda_1-c_2)\big)^N}\big(|x|^2+h\|g(x)\|^2_{\rm HS}-\alpha\big)+\alpha.
\end{align}
Because of Assumption \ref{as:f_constant}, the term on the right hand side above can be bounded by $|x|^2+\|g(x)\|^2_{\rm HS}+\alpha$, which is independent of $k$ and $h$.

\end{proof}
The next result shows two numerical solutions starting from different initial conditions can be arbitrarily close after sufficiently many iterations.
\begin{lemma}\label{lem:sta1}

Under Assumptions \ref{as:A}, \ref{as:fg} and \ref{as:f_constant}, and let $h\leq (l_1-1)/(2\lambda_1-2c)$,
 define $\hat{X}_{Nh}$ and $\hat{Y}_{Nh}$ solutions of the backward Euler-Maruyama scheme on $\mathcal{T}^h$ with different initial conditions $x,y\in U$ respectively, where $U$ is any compact subset of $\mathbb{R}^d$. Then
 \begin{equation*}
     \|\hat{X}_{Nh}-\hat{Y}_{Nh}\|  \leq \sqrt{1+c}|x-y|e^{-\xi_1Nh},
 \end{equation*}
where $\xi_1 = \frac{\lambda_1 - c}{1+2(\lambda_1 - c)}$.
\end{lemma}
\begin{proof}
Define $D_N:=\hat{X}_{Nh}-\hat{Y}_{Nh}$. Let us use \eqref{eqn:eqa} again, which allows us to examine the following term:
\begin{align*}
   &2\mathbb{E} \langle D_N-D_{N-1} ,D_N\rangle\\
   &=-2h\mathbb{E}\langle AD_N,D_N\rangle+2h\mathbb{E}\langle f\big(\hat{X}_{Nh}\big)-f\big(\hat{Y}_{Nh}\big),D_N
   \rangle\\
   &\quad +2\mathbb{E}\langle (g\big(\hat{X}_{(N-1)h}\big)-g\big(\hat{Y}_{(N-1)h}\big))\Delta W_{(N-1)h},D_N
   \rangle\\
   &\leq 2h\mathbb{E}\langle(-A+cI)D_N,D_N\rangle-l_1 h \|g\big(\hat{X}_{Nh}\big)-g\big(\hat{Y}_{Nh}\big)\|^2\\
   &\quad +2\mathbb{E}\langle (g\big(\hat{X}_{(N-1)h}\big)-g\big(\hat{Y}_{(N-1)h}\big))\Delta W_{(N-1)h},D_N-D_{N-1}
   \rangle,
\end{align*}
where we use Assumption \ref{as:fg} to deduce the last inequality and the last term is due to $$\mathbb{E}\langle (g\big(\hat{X}_{(N-1)h}\big)-g\big(\hat{Y}_{(N-1)h}\big))\Delta W_{(N-1)h},D_{N-1}
   \rangle=0.$$
This leads to
\begin{align*}
    &(1+2h(\lambda_1-c))\|D_N\|^2+l_1 h \big\|g\big(\hat{X}_{Nh}\big)-g\big(\hat{Y}_{Nh}\big)\big\|^2\\
    &\leq \|D_{N-1}\|^2+ h \big\|g\big(\hat{X}_{(N-1)h}\big)-g\big(\hat{Y}_{(N-1)h}\big)\big\|^2.
\end{align*}
Choose $h$ such that $(1+2h(\lambda_1-c))\leq l_1$, then
by iteration we have
\begin{align*}
    &\|D_N\|^2+h\big\|g\big(\hat{X}_{Nh}\big)-g\big(\hat{Y}_{Nh}\big)\big\|^2 \\
    &\leq \frac{1}{(1+2h(\lambda_1-c))^N}\big(\|D_{0}\|^2+h\big\|g\big(\hat{X}_{0}\big)-g\big(\hat{Y}_{0}\big)\big\|^2_{\rm HS} \big)\\
    &\leq \frac{1+c}{(1+2h(\lambda_1-c))^N}|x-y|^2,
\end{align*}
where we make use of Assumption \ref{as:fg} and $h\leq 1$ to deduce the last line. Since the fact that $a^N < e^{-(1-a)N}$ for any $a \in (0,1)$ and $\lambda_1>c$ in Assumption \ref{as:f_constant}, the assertion follows. 
\end{proof}

\subsection{The existence and uniqueness of the numerical invariant measure}
Now, we are ready to give the proof of Theorem \ref{thm:main2}.
\begin{proof}[The proof of Theorem \ref{thm:main2}]
Due to the Chebyshev inequality, for any initial value $x \in \mathbb{R}^d$ we obtain that $\left\{ \delta_x \hat{\mathbb{P}}_{jh} \right\}$ is tight, where $\delta_x$ is used to emphasize the initial value $x$. Then, a subsequence that converges weakly to an invariant measure $\hat{\pi} \in \mathcal{P} (\mathbb{R}^d)$ can be extracted. By the H\"older inequality and Lemma \ref{lem:sta1}, we can see that for any $r \in (0,1]$
\begin{align}\label{eq:Wx-y}
\begin{split}
    \mathbb{W}_r\left(\delta_x  \hat{\mathbb{P}}_{jh}, \delta_y \hat{\mathbb{P}}_{jh}\right) \leq \|\hat{X}_{jh}-\hat{Y}_{jh}\|_r^r \leq \|\hat{X}_{jh}-\hat{Y}_{jh}\|^r  \leq (1+c)^{r/2} |x-y|^r e^{-r\xi_1 j h}.
\end{split}
\end{align}
Then, thanks to Lemma \ref{lem:boundedness} and the Kolmogorov-Chapman equation, for any $j,l>0$ and $r\in (0,2]$ we have
\begin{align*}
\begin{split}
    \mathbb{W}_r\left(\delta_x  \hat{\mathbb{P}}_{jh}, \delta_x \hat{\mathbb{P}}_{(j+l)h} \right) &= \mathbb{W}_r\left(\delta_x \hat{\mathbb{P}}_{jh}, \delta_x \hat{\mathbb{P}}_{jh} \hat{\mathbb{P}}_{lh} \right)\\
    &\leq \int_{\mathbb{R}^d} \mathbb{W}_r\left(\delta_x  \hat{\mathbb{P}}_{jh}, \delta_y \hat{\mathbb{P}}_{jh}\right) \hat{\mathbb{P}}_{lh} (x,\mathrm{d} y)\\
    &\leq \int_{\mathbb{R}^d} (1+c)^{r/2} |x-y|^r e^{-r\xi_1 j h}\hat{\mathbb{P}}_{lh} (x,\mathrm{d} y) \\
    &\leq 2(1+c)^{r/2}\left(|x|^r + \|\hat{X}_{lh}\|^r\right) e^{-r\xi_1 j h} \\
    &\leq  K_2(r) e^{-r\xi_1 j h},
\end{split}
\end{align*}
where 
\begin{equation*}
K_2(r) := 2(1+c)^{r/2}\left(|x|^r + \left( |x|^2 + \|g(x)\|^2_{\rm HS} +  \frac{c_1}{\lambda_1-c_2}\right)^{r/2}\right).
\end{equation*}
Thus, letting $l \rightarrow \infty$ indicates
\begin{equation*}
 \mathbb{W}_r\left(\delta_x  \hat{\mathbb{P}}_{jh}, \hat{\pi} \right) \leq   K_2(r) e^{-r\xi_1 j h}.
\end{equation*}
Moreover, we have
\begin{equation*}
    \mathbb{W}_r\left(\delta_x  \hat{\mathbb{P}}_{jh}, \hat{\pi} \right) \rightarrow 0, \quad \text{as}~j \rightarrow \infty,
\end{equation*}
which guarantees that $\hat{\pi}$ is the unique invariant measure of $\left\{ \delta_x \hat{\mathbb{P}}_{jh} \right\}$. Now, assume that $\hat{\pi}_1 \in \mathcal{P} (\mathbb{R}^d)$ is the invariant measure of $\hat{X}_{jh}$ with the initial value $x$ and $\hat{\pi}_2 \in \mathcal{P} (\mathbb{R}^d)$ is the invariant measure of $\hat{X}_{jh}$ with the initial value $y$, we can see 
\begin{equation*}
\mathbb{W}_r \left(\hat{\pi}_1, \hat{\pi}_2 \right) \leq \int_{\mathbb{R}^d \times \mathbb{R}^d} \mathbb{W}_r \left(\delta_x \hat{\mathbb{P}}_{jh},  \delta_y \hat{\mathbb{P}}_{jh} \right) \nu (\mathrm{d} x, \mathrm{d} y) ,
\end{equation*}
for any $x,y \in \mathbb{R}^d$ with $x \neq y$. 
Therefore, by \eqref{eq:Wx-y} the BEM method has a unique invariant measure.
\end{proof}


\subsection{The uniform strong convergence of the BEM method}
The proof of Theorem \ref{thm:error} is presented as follows.
\begin{proof}[The proof of Theorem \ref{thm:error}]
First note that
\begin{align}
\begin{split}
        X_{Nh}&=X_{(N-1)h}-\int_{(N-1)h}^{Nh}AX_{s}\mathrm{d}s+\int_{(N-1)h}^{Nh}f(X_{s})\mathrm{d}s+\int_{(N-1)h}^{Nh}g(X_{s})\mathrm{d}W_s\\
    &=X_{(N-1)h}-\int_{(N-1)h}^{Nh}A\big(X_{s}-X_{Nh}\big)\mathrm{d}s-hAX_{Nh}\\
    &\quad +\int_{(N-1)h}^{Nh}(f(X_{s})-f(X_{Nh}))\mathrm{d}s+hf(X_{Nh})\\
    &\quad +\int_{(N-1)h}^{Nh}\Big(g(X_{s})-g(X_{(N-1)h}))\mathrm{d}W_s+g\big(X_{(N-1)h}\big)\Delta W_{(N-1)h}.
\end{split}
\end{align}
Define $e_N:=X_{Nh}-\hat{X}_{Nh}$. Then 
\begin{align*}
   &2\mathbb{E} \langle e_N-e_{N-1} ,e_N\rangle\\
   &=-2h\mathbb{E}\langle Ae_N,e_N\rangle+2h\mathbb{E}\langle f(X_{Nh})-f(\hat{X}_{Nh}),e_N
   \rangle\\
   &\quad +2\mathbb{E}\Big\langle-\int_{(N-1)h}^{Nh}A(X_{s}-X_{Nh})\mathrm{d}s ,e_N\Big\rangle\\
   &\quad+2\mathbb{E}\Big\langle\int_{(N-1)h}^{Nh}(f(X_{s})-f(X_{Nh}))\mathrm{d}s ,e_N\Big\rangle\\
   &\quad+2\mathbb{E}\Big\langle\int_{(N-1)h}^{Nh}(g(X_s)-g\big(X_{(N-1)h}))\mathrm{d}W_s ,e_N\Big\rangle\\
   &\quad+2\mathbb{E}\langle (g(X_{(N-1)h})-g(\hat{X}_{(N-1)h}))\Delta W_{(N-1)h},e_N
   \rangle.
  \end{align*}
Note that for $t\in [(N-1)h,Nh]$, $\int_{(N-1)h}^{t}e_{N-1}^T(g(X_s)-g\big(X_{(N-1)h}))\mathrm{d}W_s$ gives a martingale, where $a^T$ represents the transpose of a vector or matrix $a$. To see it, define the stopping time $\tau_{N,K}:=\inf\{s: |X_s|>K+|X_{(N-1)h}|\}$. Note that $\{\tau_{N,K}\}_{K\in \mathbb{N}}$ is nondescreasing and $\lim_{K\to \infty}\tau_{N,K}=\infty$. Then one can check that $\int_{(N-1)h}^{t\wedge  \tau_{N,K}}e_{N-1}^{T}(g(X_s)-g\big(X_{(N-1)h}))\mathrm{d}W_s$ is indeed a martingale. Then we have $$\mathbb{E}\Big\langle\int_{(N-1)h}^{Nh}(g(X_s)-g\big(X_{(N-1)h}))\mathrm{d}W_s ,e_N\Big\rangle=\mathbb{E}\Big\langle\int_{(N-1)h}^{Nh}(g(X_s)-g\big(X_{(N-1)h}))\mathrm{d}W_s ,e_N-e_{N-1}\Big\rangle.$$

By Young's inequality 
$$2ab\leq \epsilon^2a^2+\frac{b^2}{\epsilon^2}, \quad \forall a,b>0,$$
and Assumption \ref{as:fg}, we are able to choose $\epsilon_0^2:=h(\lambda_1-c)/2$ such that
\begin{align*}
&2\mathbb{E} \langle e_N-e_{N-1} ,e_N\rangle\\
     &\leq 2h\mathbb{E}\langle(-A+cI)e_N,e_N\rangle-hl_1\|g(X_{Nh})-g(\hat{X}_{Nh})\|^2\\
   &\quad +2\epsilon_0^2 \|e_N\|^2+\frac{1}{\epsilon^2_0}\Big\|-\int_{(N-1)h}^{Nh}A(X_{s}-X_{Nh})\mathrm{d}s\Big\|^2\\
   &\quad +\frac{1}{\epsilon^2_0}\Big\|\int_{(N-1)h}^{Nh}(f(X_{s})-f(X_{Nh}))\mathrm{d}s\Big\|^2\\
   &\quad +2\Big\|\int_{(N-1)h}^{Nh}(g(X_{s})-g(X_{(N-1)h}))\mathrm{d}W_s\Big\|^2\\
   &\quad +2h\|g(X_{(N-1)h})-g(\hat{X}_{(N-1)h})\|^2+\|e_N-e_{N-1}\|^2.
\end{align*}
By Proposition \ref{prop:use_bound}, we know there exists a constant $C$ depending on $q$, $A$, $f$ and $g$ such that
\begin{align*}
    &\Big\|-\int_{(N-1)h}^{Nh}A\big(X_{s}-X_{-k\tau+Nh}\big)\mathrm{d}s\Big\|^2\\
   &+\Big\|\int_{(N-1)h}^{Nh}\Big(f(X_{s})-f(X_{Nh})\Big)\mathrm{d}s\Big\|^2\\
   &\leq Ch^3\Big(1+\sup_{s\geq 0} \|X_{s}\|_{4q-2}^{2q-1}\Big):=\beta h^3.
\end{align*}
Besides, by the It\^o isometry and the H\"older continuity of $X$ in temporal variable as shown in Proposition \ref{prop:use_bound},
\begin{align*}
    &2\Big\|\int_{(N-1)h}^{Nh}(g(X_{s})-g(X_{(N-1)h}))\mathrm{d}W_s\Big\|^2\leq \frac{2c}{l_1}h^2:=\hat{c}h^2.
\end{align*}
Note that $\beta$ is bounded because of Proposition \ref{prop:use_bound}. Define $G_N:=\|g(X_{Nh})-g(\hat{X}_{Nh})\|^2$.
Then from \eqref{eqn:eqa} and the estimate above we have that
\begin{align*}
     & \|e_N\|^2-\|e_{N-1}\|^2+hl_1 \|G_N\|^2\leq 2\mathbb{E} \langle e_N-e_{N-1} ,e_N\rangle\\
   &\leq 2h\mathbb{E}\langle(-A+cI)e_N,e_N\rangle+2\epsilon_0^2 \|e_N\|^2+\frac{\beta h^3}{\epsilon^2_0}+\hat{c}h^2+2h\|G_{N-1}\|^2.
\end{align*}
Define $\hat{\alpha}:=\frac{2\beta +c(\lambda_1-c)}{(\lambda_1-c)^2}h$. Since that $1+h(\lambda_1-c)\leq l_1/2$, then the inequality above can be rearranged to
\begin{align*}
    \Big(1+h(\lambda_1-c)\Big)\big(\|e_N\|^2+2h\|G_N\|^2-\hat{\alpha}\big)\leq \|e_{N-1}\|^2+2h\|G_{N-1}\|^2-\hat{\alpha}.
\end{align*}

By iteration we have 
$$\|e_N\|^2+2h\|G_N\|^2\leq \Big(1-\frac{1}{1+h(\lambda_1-c)^N}\Big)\frac{2\beta +c(\lambda_1-c)}{(\lambda_1-c)^2}h,$$
 because $X_0=\hat{X}_0$.
Finally due to Assumption \ref{as:f_constant}, we have $\|e_N\|^2\leq \frac{2\beta +c(\lambda_1-c)}{(\lambda_1-c)^2}h.$ Then the assertion follows.
\end{proof}

\subsection{Convergence of the numerical invariant measure to the underlying counterpart}
Now we are ready to show the last main theorem.
\begin{proof}[The proof of Theorem \ref{thm:numinvconv}]
By the triangle inequality, we have
\begin{equation*}
    \mathbb{W}_r\left(\hat{\pi},\pi  \right) \leq \mathbb{W}_r\left(\hat{\pi},\delta_x \hat{\mathbb{P}}_{jh}  \right) + \mathbb{W}_r\left(\delta_x \mathbb{P}_{jh},\pi  \right) +  \mathbb{W}_r\left(\delta_x \mathbb{P}_{jh},\delta_x \hat{\mathbb{P}}_{jh}  \right).
\end{equation*}
Thanks to Theorems \ref{thm:main2} and \ref{thm:trueinvmea}, the convergences of $\hat{\mathbb{P}}_{jh}$ to $\hat{\pi}$  and $\mathbb{P}_{jh}$ to $\pi$ yield
\begin{equation*}
    \mathbb{W}_r\left(\hat{\pi},\delta_x \hat{\mathbb{P}}_{jh}  \right) \leq K_2(r) e^{-r\xi_1 j h} \quad \text{and} \quad    \mathbb{W}_r\left(\delta_x \mathbb{P}_{jh},\pi  \right) \leq C e^{-\xi_2 jh},
\end{equation*}
where $C$ is a genetic constant in this proof that may be different from line to line. Now, for any fixed $h \in (0,h^{*}\wedge h^{**})$, there is a sufficient large $j^{*}$ such that for any $j > j^{*}$
\begin{equation*}
    \mathbb{W}_r\left(\hat{\pi},\delta_x \hat{\mathbb{P}}_{jh}  \right) \leq C h^{r/2} \quad \text{and} \quad    \mathbb{W}_r\left(\delta_x \mathbb{P}_{jh},\pi  \right) \leq C h^{r/2}.
\end{equation*}
Then for the fixed $j>j^{*}$, we derive from Theorem \ref{thm:error} that
\begin{equation*}
    \mathbb{W}_r\left(\delta_x \mathbb{P}_{jh},\delta_x \hat{\mathbb{P}}_{jh}  \right) \leq C h^{r/2}.
\end{equation*}
Therefore, the assertion is proved.
\end{proof}

\section{Numerical examples}
In this section, two numerical examples are presented. Example \ref{example:1d} is used to illustrate that the BEM method admits a unique invariant measure, which then converges to the underlying one. In Example \ref{example:2d}, we discuss the application of our numerical method in the stabilisation of SDEs in the distribution sense.  
\begin{example}\label{example:1d}
Consider a scalar mean-reverting type model with super-linear coefficients 
\begin{equation*}
    dX_t = \left( b - \alpha X_t - \beta X_t^3\right) \mathrm{d}{t} + \sigma X_t^2 \mathrm{d}{W_t},\quad X_0 =x.
\end{equation*}
\end{example}
By setting $b=1$, $\alpha=1$, $\beta=2$ and $\sigma=1$, it is not hard to see that all the assumptions are satisfied. Therefore, according to our theorems there exists a unique invariant measure for the BEM method. One thousand sample paths are simulated with $X_0=5$ and $h=0.01$, which are then used to construct empirical density functions at different time points. It is clear to see from the left plot in  Figure \ref{fig1} that the shapes of empirical density functions at $t=0.1$, $t=0.3$ and $t=0.5$ are quite different but the ones at $t=4$ and $t=10$ are much more similar, which indicates the existence of the invariant measure. From the right plot in Figure \ref{fig1}, we can see the empirical density functions at the same time point $t=90$ but with different initial values $-5,5,15$ are quite close to each other, which indicates uniqueness of the invariant measure.
\begin{figure}
\includegraphics[width=0.475\textwidth]{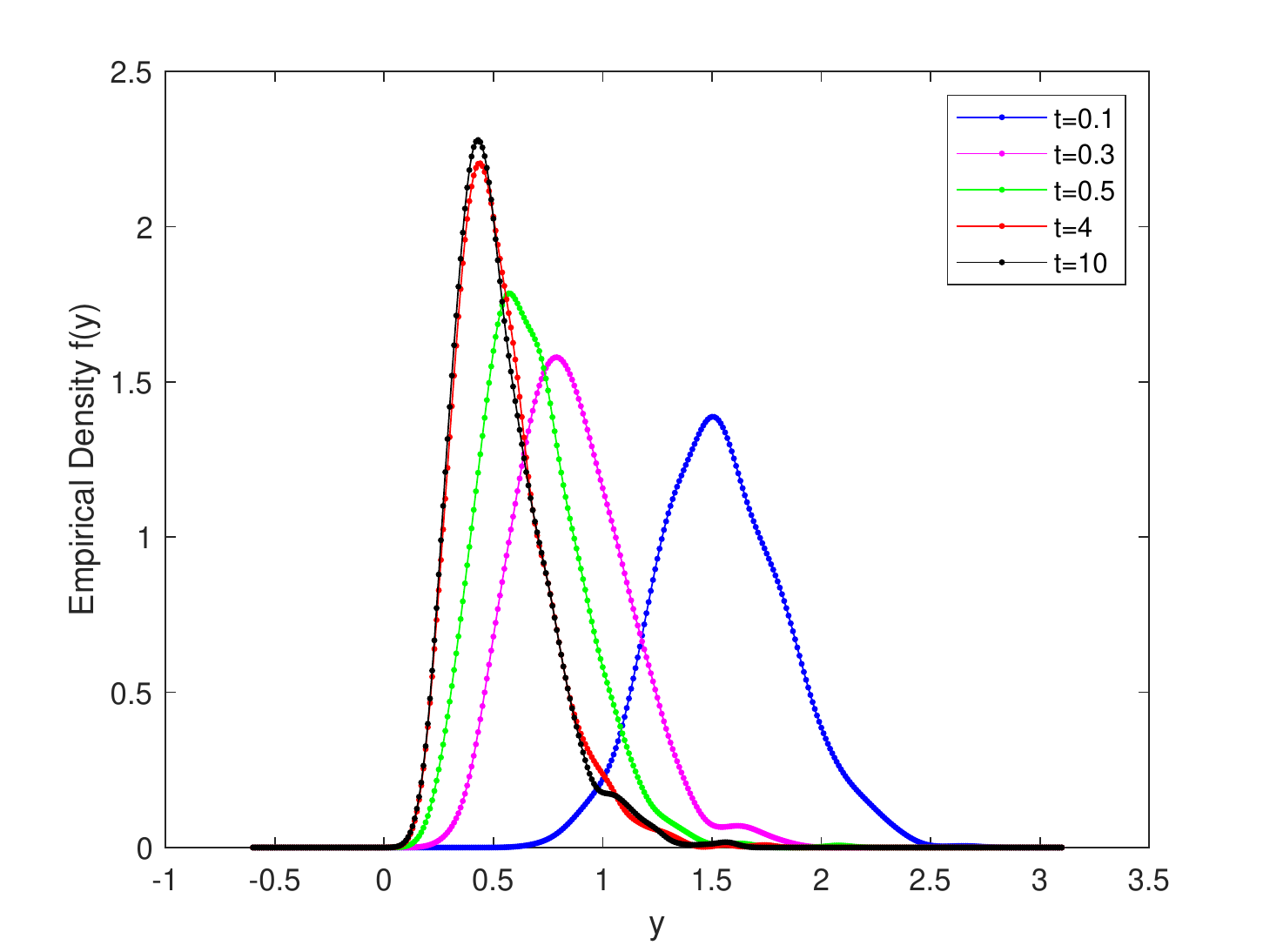}
\hspace{\fill}
\includegraphics[width=0.475\textwidth]{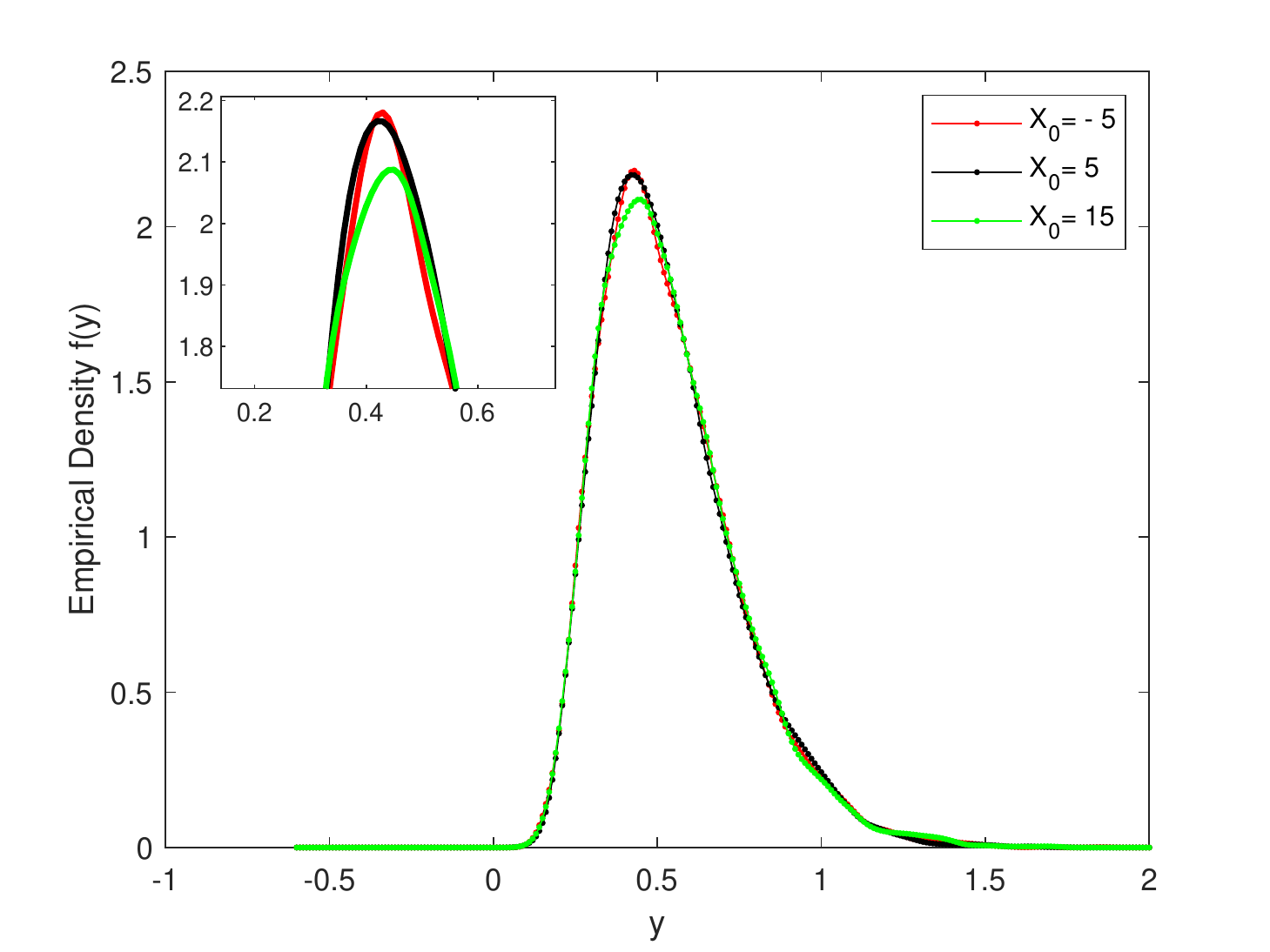}
\caption{Left: Empirical density functions at different time points. Right: Empirical density functions at t=90 with different initial values }
\label{fig1}       
\end{figure}
To measure the difference between empirical density functions at consecutive time points $t=ih$ and $t=(i+1)h$ for $i=0,1,...$, the Kolmogorov–Smirnov (K-S) test is employed to test a sequence of hypotheses that
$$H_0: \quad \text{Two samples at t=ih and t=(i+1)h are from the same distribution},$$  
$$H_1: \quad \text{Two samples at t=ih and t=(i+1)h are from different distributions},$$
for $i=0,1,...200$. It can be observed from the upper plot in Figure \ref{fig2} that as time gets large the differences between empirical density functions at consecutive time points vanish, which indicates the existence of the invariant measure for the numerical solution. The lower plot in Figure \ref{fig2} also confirms this conclusion as the p values are quite close to 1 as time advances.
\begin{figure}
\centering
  \includegraphics[width=0.8\textwidth]{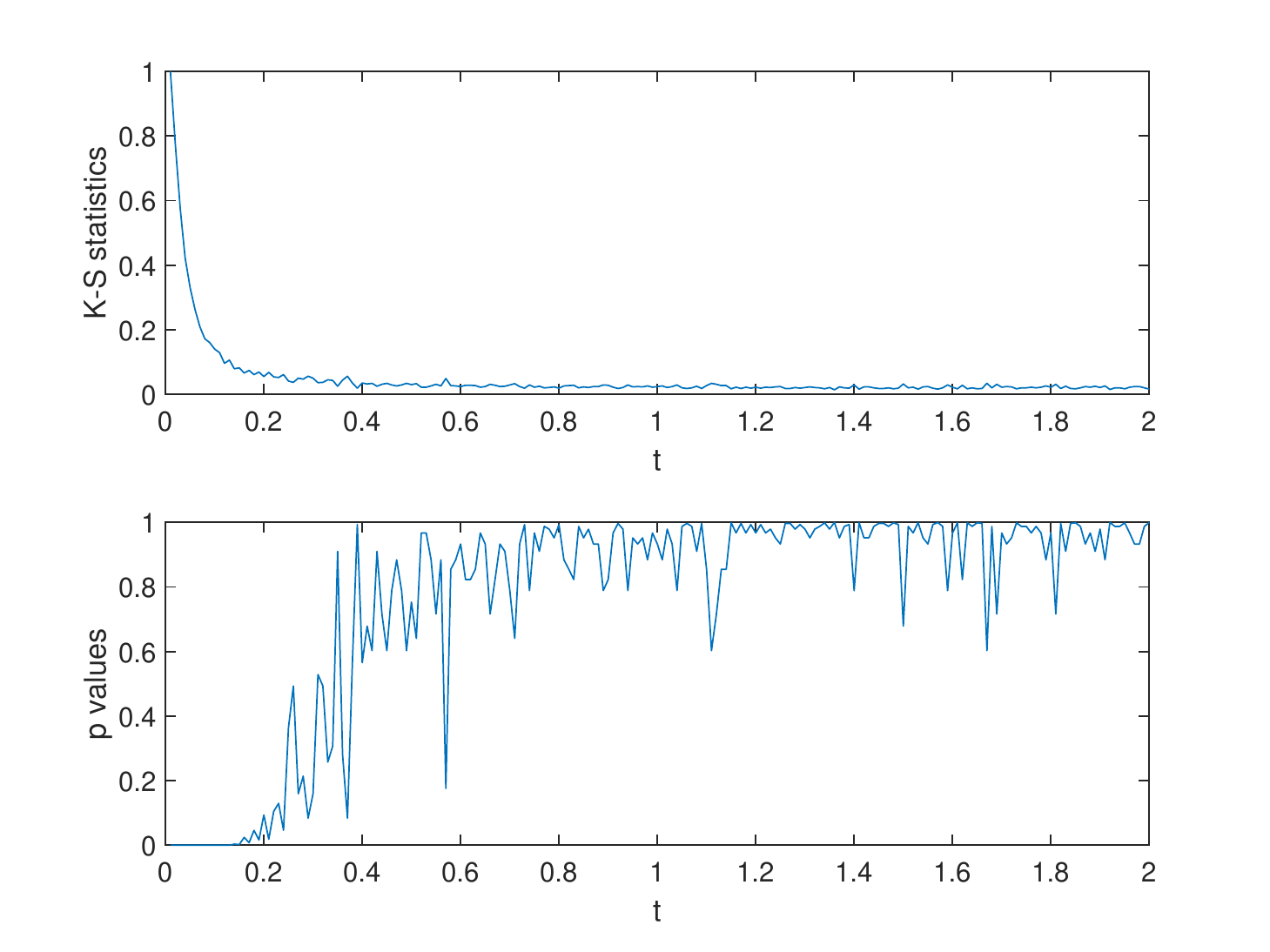}
\caption{K-S tests for samples at consecutive time points for Example \ref{example:1d}}
\label{fig2}       
\end{figure}

Now we turn to our second example, which could be regarded as an illustration of the application of our results in the system control problem. To make it clear, we brief the problem as follows.

In the very recent works \cite{LLLM2022,YHLM2022}, the authors discussed the design of some controllers to stabilise some SDEs that originally are not stable in distribution. To be more precise, for some unstable SDE (i.e. not stable in the distribution sense)
\begin{align*}
  \begin{cases}
    \mathrm{d}{X_t} = 
    f(X_t)  \mathrm{d}{t}+g(X_t)\mathrm{d}{W_t},& \quad \text{for } t >0,\\
    X_{0} = x\in \mathbb{R}^d,& 
  \end{cases}
\end{align*}
the authors in those two works used some past state $X_{t-\tau}$, where the small enough constant $\tau>0$ represents the time delay, to design a controller $AX_{t-\tau}$ such that the controlled system
\begin{align}
  \label{eq:ctlSDE}
  \begin{cases}
    \mathrm{d}{X_t} = 
    \big[f(X_t) - A X_{t-\tau}\big] \mathrm{d}{t}+g(X_t)\mathrm{d}{W_t},& \quad \text{for } t >0,\\
    X_{0} = x\in \mathbb{R}^d& 
  \end{cases}
\end{align}
is stable in distribution. In their works, the authors proposed the method to design the controller and proved theoretically that the controlled system is indeed stable in distribution. But in practice, numerical methods are always required for the applications of those theorems, as the explicit forms of the true solutions of stochastic systems can hardly be found, not to mention the explicit forms of the invariant distributions. Therefore, trusted numerical methods are essential for demonstrating those theorems in \cite{LLLM2022,YHLM2022} and displaying the shapes of the invariant distributions. By saying trusted numerical methods, we mean those methods that have been proved to be able to approximate the underlying true invariant distributions. And this is what we proved in this paper for the BEM method. 

It is clear that if the $AX_{t-\tau}$ is replaced by $AX_{t}$ in the controlled system \eqref{eq:ctlSDE}, then it looks exactly like the SDE \eqref{eq:SPDE} studied in this paper. Since our results obtained in this paper do not include delay terms in the equations, we use $AX_{t}$ as the controller in our Example \ref{example:2d}. In future, We are going to work out the numerical invariant measures for some stochastic delay differential equations.
\begin{example}\label{example:2d}
Consider a two dimensional SDE  
\begin{align*}
  \begin{cases}
    \mathrm{d}{X_{t,1}} = 
    \big[10 + 2 X_{t,1} - X_{t,2} \big] \mathrm{d}{t}+\big[0.5 + 0.1X_{t,2}  \big]\mathrm{d}{W_{t,1}}, &\\
    \mathrm{d}{X_{t,2}} = 
    \big[5 + X_{t,1} + 3 X_{t,2} - X_{t,2}^3 \big] \mathrm{d}{t}+\big[0.3 + 0.1\big( X_{t,1} + X_{t,2}^2\big)  \big]\mathrm{d}{W_{t,2}},& \\
    X_{0} = \big(5,5\big)
  \end{cases}
\end{align*}
which is unstable in distribution for any initial data. According to theorems in \cite{LLLM2022,YHLM2022}, one can design a controller
\[ A = \begin{pmatrix}
-5 & 0 \\
-2 & -4
\end{pmatrix} \]
such that the controlled system
\begin{align}
  \label{eq:2dexample-stable}
  \begin{cases}
    \mathrm{d}{X_{t,1}} = 
    \big[10 - 3 X_{t,1} - X_{t,2} \big] \mathrm{d}{t}+\big[0.5 + 0.1X_{t,2}  \big]\mathrm{d}{W_{t,1}} &\\
    \mathrm{d}{X_{t,2}} = 
    \big[5 - X_{t,1} - 3 X_{t,2} - X_{t,2}^3 \big] \mathrm{d}{t}+\big[0.3 + 0.1\big( X_{t,1} + X_{t,2}^2\big)  \big]\mathrm{d}{W_{t,2}}& 
  \end{cases}
\end{align}
is stable in the distribution sense. But, in practice one may further ask the question: what does the unique distribution look like?
\end{example}
To answer the question, one may turn to our results in this paper. Since it is not hard to check that coefficients of \eqref{eq:2dexample-stable} satisfy the requirements, we can regard the numerical invariant distribution generated by the BEM method as a trusted approximate to the underlying one. 1000 sample paths generated by the BEM method with the step size of 0.05 are simulated. Similar to Example \ref{example:1d}, the K-S test is applied to illustrate that the distributions generated by the BEM method indeed tends to a unique one as the time advances. The asymptotic behaviour of the K-S statistics in Figure \ref{fig:2dkstest} confirms it. More importantly, Figure \ref{fig:2dkstest} also indicates that one does not have to simulate sample paths for long time to see the invariant distribution, as the differences between empirical distributions decay to zero in a quite fast way.       
\begin{figure}
\centering
  \includegraphics[width=0.5\textwidth]{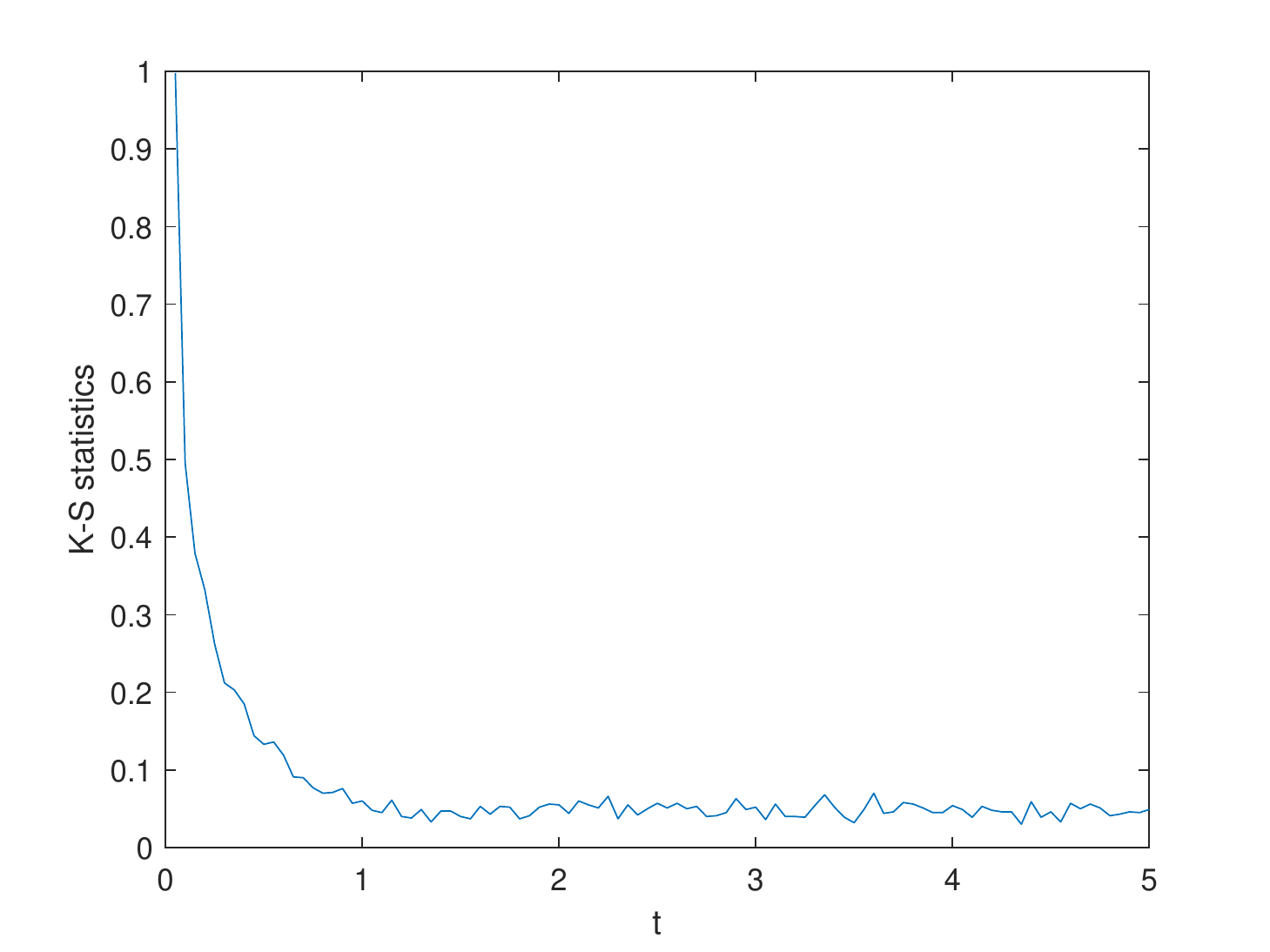}
\caption{K-S tests for samples at consecutive time points for Example \ref{example:2d}}
\label{fig:2dkstest}       
\end{figure}
Therefore, to see the shape of the unique distribution of \eqref{eq:2dexample-stable}, it is sufficient to use the empirical distribution of the numerical solutions generated by the BEM method at relatively small time point. Figure \ref{fig:2dexample3dplot} displays the empirical density function of the solution $(X_{t,1},X_{t,2})$ at $t=4$, which could be used to answer the question raised in Example \ref{example:2d}. In practice, one can further use some non-parametric and parametric approaches to find out what the distribution is and the estimated values of parameters of it.
\begin{figure}
\centering
\includegraphics[width=0.7\textwidth]{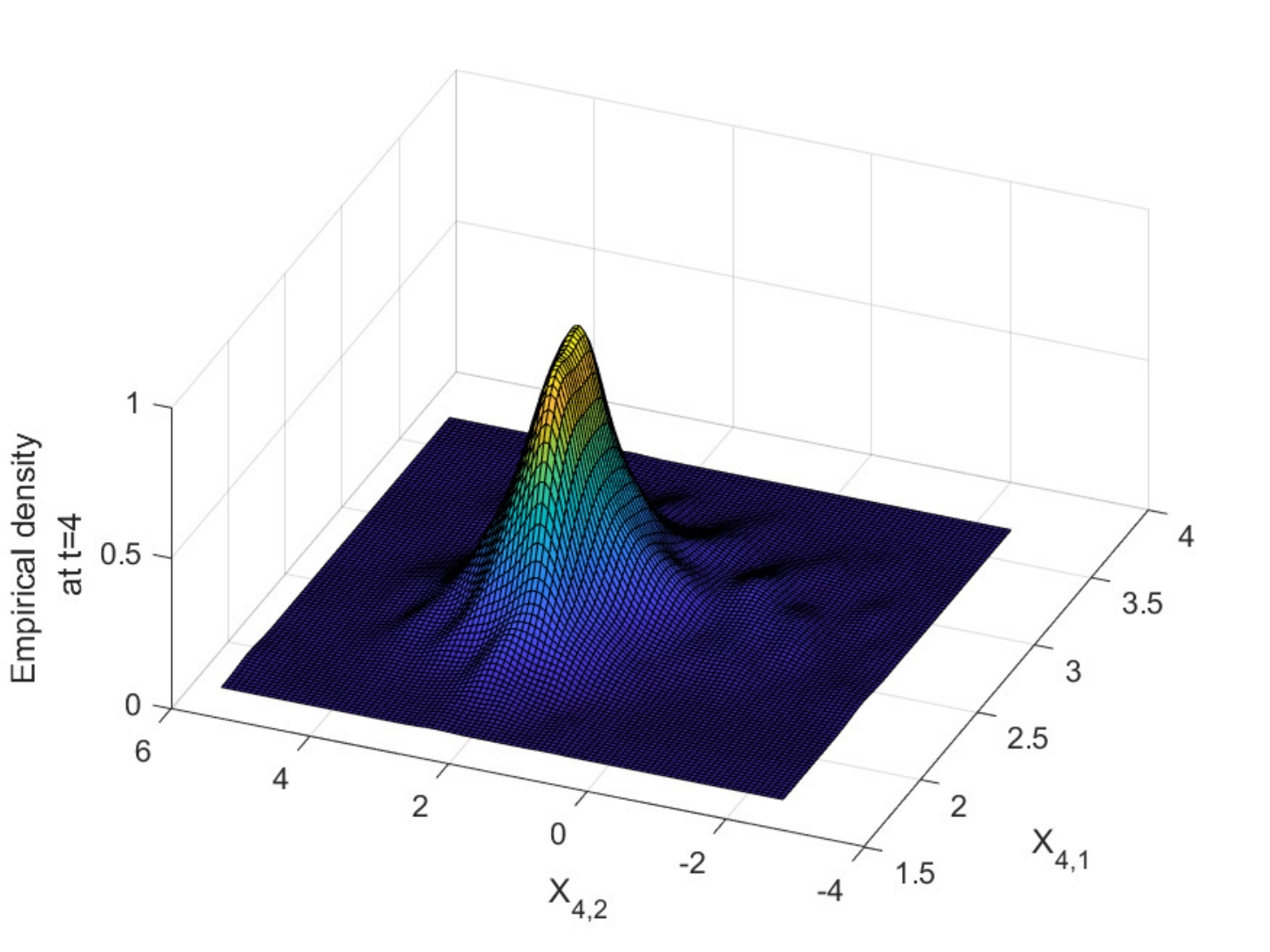}
\caption{Empirical density function at $t=4$}
\label{fig:2dexample3dplot}       
\end{figure}

 To end up this section, we give a short informal discussion on the potential application of our results in numerical approximates to stationary Fokker-Planck equations. It is well know that if there exists a unique invariant measure $\pi$ for the SDE
\begin{equation*}
    \mathrm{d}{X_t} = \mu(X_t)  \mathrm{d}{t}+\sigma(X_t)\mathrm{d}{W_t},
\end{equation*}
then the true $\pi$ can be found by solving the following partial differential equation (PDE)
\begin{equation}\label{eq:F-Ppde}
    \frac{1}{2} \frac{\partial^2}{\partial x^2} \big( \sigma^2 \pi \big) - \frac{\partial}{\partial x} \big( \mu \pi\big) = 0 \quad \text{with the condition}\quad  \int_{x \in \mathbb{R}^d} \pi(x) \mathrm{d}{x}=1.
\end{equation}
The numerical invariant measure $\hat{\pi}$ obtained in this paper can be regarded as a good estimator for the solution of \eqref{eq:F-Ppde}, as the convergence of $\hat{\pi}$ to $\pi$ actually has been proved in this paper. For example, Figure \ref{fig:2dexample3dplot} indeed display the solution to the stationary Fokker-Planck equation that is corresponding to the SDE \eqref{eq:2dexample-stable}. Fokker-Planck equations and their stationary forms are of importance on their own rights in various problems arising in chemical reactions, statistical physics, and fluid mechanics, however, their practical use is hindered by the curse of dimensionality. Based on the success of \cite{HJE2018}, it is expected that under some smart design of neural network architecture through the probabilistic representation and the numerical simulation, one may establish an effective stochastic framework for the PDE \eqref{eq:F-Ppde}, which could avoid the curse of dimensionality.

\section{Conclusion and future research}
In this paper, we revisited the classical BEM method and showed the existence and uniqueness of its invariant measure when both the drift and the diffusion coefficients are allowed to contain some super-linear terms. In addition, the convergence of the numerical invariant measure to its underlying counterpart was also proved. Numerical simulations were provided to demonstration our theorems and their potential applications in system controls. 

As we mentioned occasionally in this paper, there are many works that have not been done in this area. One definitely interesting work is to extend the results in this paper to stochastic delay differential equations, for which the concept of invariant measure is quite different from the case of SDEs. Another question that is worth to be considered is the numerical invariant measure of hybrid SDEs with super-linear drift and diffusion coefficients, in which the switches among different modes would play important roles in the stability in distribution of the whole system. 

\section*{Acknowledgements}
Wei Liu would like to thank Shanghai Rising-Star Program (Grant No. 22QA1406900), Science and Technology Innovation Plan of Shanghai (Grant No. 20JC1414200), and the National Natural Science Foundation of China (Grant No. 11871343, 11971316) for their financial support.

\bibliographystyle{plain}

\bibliography{bibfile}

\begin{thebibliography}{100}

\bibitem{Allen2007} E. Allen, 
\newblock{\em Modeling with It\^o Stochastic Differential Equations,} 
Springer, Dordrecht, 2007.

\bibitem{AK2017} A. Andersson and R. Kruse, 
\newblock Mean-square convergence of the BDF2-Maruyama and backward Euler schemes for SDE satisfying a global monotonicity condition, 
\newblock{\em BIT,} 57.1(2017), 21-53.

\bibitem{BSY2016} J. Bao, J. Shao, and C. Yuan, 
\newblock Approximation of invariant measures for regime-switching
diffusions, 
\newblock{\em Potential Anal.,} 44(2016), 707-727.


\bibitem{kruse2016} W. J. Beyn,  E. Isaak, and R. Kruse,
\newblock  Stochastic C-stability and B-consistency of explicit and implicit Euler-type schemes.
\newblock{\em J. Sci. Comput.,} 67.3(2016), 955-987.

\bibitem{CG2020} Z. Chen and S. Gan, 
\newblock Convergence and stability of the backward Euler method for jump–diffusion SDEs with super-linearly growing diffusion and jump coefficients, 
\newblock{\em J. Comput. Appl. Math.,} 363(2020), 350-369.


\bibitem{FG2018} W. Fang and M.B. Giles, 
\newblock Adaptive Euler–Maruyama method for SDEs with non-globally lipschitz drift, 
\newblock{\em Springer Proceedings in Mathematics and Statistics,} 241(2018), 217-234.

\bibitem{GLM2019} J. Gao, H. Liang and S. Ma, 
\newblock Strong convergence of the semi-implicit Euler method for nonlinear stochastic Volterra integral equations with constant delay, 
\newblock{\em Appl. Math. Comput.,} 348(2019), 385-398.

\bibitem{HJE2018} J. Han, A. Jentzen and W. E, 
\newblock Solving high-dimensional partial differential equations using deep learning, 
\newblock{\em Proc. Natl. Acad. Sci. USA,} 115.34(2018), 8505-8510.

\bibitem{minskii1980} R. Z. Has'minskii, 
\newblock{\em  Stochastic Stability of Differential Equations,} Sijthoff \& Noordhoff, 1980.

\bibitem{Higham2011} D.J. Higham, 
\newblock Stochastic ordinary differential equations in applied and computational mathematics, 
\newblock{\em IMA J. Appl. Math.,} 76.3(2011), 449-474.


\bibitem{HJK2012} M. Hutzenthaler, A. Jentzen and P.E. Kloeden, 
\newblock Strong convergence of an explicit numerical method for SDEs with nonglobally lipschitz continuous coefficients, 
\newblock{\em Ann. Appl. Probab.,} 22.4(2012), 1611-1641.

\bibitem{JWL2020} Y. Jiang, L. Weng and W. Liu, 
\newblock Stationary distribution of the stochastic theta method for nonlinear stochastic differential equations, 
\newblock{\em Numer. Algorithms,} 83.4(2020), 1531-1553.

\bibitem{LLLM2022} X. Li, W. Liu, Q Luo and X. Mao, 
\newblock Stabilisation in distribution of hybrid stochastic differential equations by feedback control based on discrete-time state observations, 
\newblock{\em Automatica J. IFAC,} 140(2022), 110210.


\bibitem{LMY2019} X. Li, X. Mao and G. Yin, 
\newblock Explicit numerical approximations for stochastic differential equations in finite and infinite horizons: Truncation methods, convergence in pth moment and stability, 
\newblock{\em IMA J. Numer. Anal.,} 39.2(2019), 847-892.

\bibitem{LM2015} W. Liu and X. Mao,  
\newblock Numerical stationary distribution and its convergence for nonlinear stochastic differential equations, 
\newblock{\em J. Comput. Appl. Math.,} 276(2015), 16-29.

\bibitem{Mao2008} X. Mao, 
\newblock{\em Stochastic Differential Equations and Applications,} second ed., Horwood, 2008.

\bibitem{Mao2015} X. Mao, 
\newblock The truncated Euler-Maruyama method for stochastic differential equations, 
\newblock{\em J. Comput. Appl. Math.,} 290(2015), 370-384.

\bibitem{MS2013} X. Mao and L. Szpruch,
\newblock Strong convergence rates for backward Euler-Maruyama method for non-linear dissipative-type stochastic differential equations with super-linear diffusion coefficients, 
\newblock{\em Stochastics,} 85.1(2013), 144-171.

\bibitem{MSH2002} J.C. Mattingly, A.M. Stuart and D.J. Higham,
\newblock Ergodicity for SDEs and approximations: Locally Lipschitz vector fields and degenerate noise, 
\newblock{\em Stochastic Process. Appl.,} 101.2(2002), 185-232.

\bibitem{MT2010} G.N. Milstein and M.V. Tretyakov, 
\newblock{\em Stochastic Numerics for Mathematical Physics,} 
\newblock Springer-Verlag, Berlin, Heidelberg, 2010. 




\bibitem{ortega2000} J. M. Ortega and W. C. Rheinboldt,
\newblock{\em  Iterative solution of nonlinear equations in several variables, volume 30 of Classics in Applied Mathematics,} Society for Industrial and Applied Mathematics (SIAM),
Philadelphia, PA, 2000. Reprint of the 1970 original.


\bibitem{stuart1996} A.M. Stuart and A.R. Humphries,
\newblock{\em  Dynamical Systems and Numerical Analysis, volume 2 of Cambridge
Monographs on Applied and Computational Mathematics,}
\newblock  Cambridge University Press, Cambridge, 1996.

\bibitem{Talay1990} D. Talay, 
\newblock Second-order discretization schemes of stochastic differential systems for the computation of the invariant law, 
\newblock{\em Stochastics,} 29.1(1990), 13-36.

\bibitem{van2007} N. G. Van Kampen,
\newblock{\em Stochastic Processes in Physics and Chemistry,}
\newblock  Elsevier, 2007.



\bibitem{WWD2020} X. Wang, J. Wu and B. Dong, 
\newblock Mean-square convergence rates of stochastic theta methods for SDEs under a coupled monotonicity condition, 
\newblock{\em BIT,} 60.3(2020), 759-790.

\bibitem{wanner1996} G. Wanner, and E. Hairer,
\newblock{\em Solving ordinary differential equations II,}  Vol. 375 (1996). Springer Berlin Heidelberg.

\bibitem{weiss2009} B. Weiss and E. Knoblock,
\newblock A stochastic return map for stochastic differential equations,
\newblock{\em  J. Stat. Phys.,} 58(1990), 863-883.

\bibitem{WL2019} L. Weng and W. Liu, 
\newblock Invariant measures of the Milstein method for stochastic differential equations with commutative noise, 
\newblock{\em Appl. Math. Comput.,} 358(2019), 169-176.

\bibitem{willett1965} D. Willett and J.S.W. Wong, 
\newblock   On the discrete analogues of some generalizations of Gr\"onwall's inequality,
\newblock{\em  Monatshefte f\"ur Mathematik,} 69.4 (1965), 362-367.


\bibitem{W2022} Y. Wu,
\newblock  Backward Euler–Maruyama method for the random periodic solution of a stochastic differential equation with a monotone drift, 
\newblock{\em J. Theor. Probab.,} (2022), 1-18.


\bibitem{XSL2011} Y. Xiao, M. Song and M. Liu, 
\newblock Convergence and stability of the semi-implicit Euler method with variable stepsize for a linear stochastic pantograph differential equation, 
\newblock{\em Int. J. Numer. Anal. Model.,} 8.2(2011), 214-225.


\bibitem{yevik2011} A. Yevik, and H. Zhao, 
\newblock  Numerical approximations to the stationary solutions of stochastic differential equations,
\newblock{\em  SIAM J. Numer. Anal.,} 49.4.13 (2011), 97-1416.

\bibitem{YHLM2022} S. You, L. Hu, J. Lu and X. Mao
\newblock Stabilization in Distribution by Delay Feedback Control for Hybrid Stochastic Differential Equations, 
\newblock{\em IEEE Trans. Automat. Control,} 67.2(2022), 971-977.


\bibitem{YM2004} C. Yuan and X. Mao
\newblock Stability in distribution of numerical solutions for stochastic differential equations, 
\newblock{\em Stoch. Anal. Appl.,} 22.5(2004), 1133-1150.

\bibitem{ZX2019} C. Zhang and Y. Xie, 
\newblock Backward Euler-Maruyama method applied to nonlinear hybrid stochastic differential equations with time-variable delay, 
\newblock{\em Sci. China Math.,} 62.3(2019), 597-616

\bibitem{Zhou2015} S. Zhou, 
\newblock Strong convergence and stability of backward Euler–Maruyama scheme for highly nonlinear hybrid stochastic differential delay equation, 
\newblock{\em Calcolo,} 52.4(2015), 445-473.


\end{thebibliography}

\end{document}